\documentclass[letterpaper, 10 pt, conference]{ieeeconf}  

\IEEEoverridecommandlockouts                              
\usepackage{graphics} 
\usepackage{epsfig} 
\usepackage{mathptmx} 
\usepackage{times} 
\usepackage{amsmath} 
\usepackage{amssymb}  

\usepackage{physics}
\usepackage{bm}
\usepackage{amsfonts}
\newtheorem{theorem}{Theorem}
\newtheorem{proposition}{Proposition}
\newtheorem{lemma}{Lemma}
\newtheorem{definition}{Definition}
\newtheorem{remark}{Remark}
\newtheorem{example}{Example}
\usepackage{breqn}
\usepackage{mathrsfs}

\newcommand\condExpectA[2]{\mathbb{E}\left[\left.#1\right|#2\right]}

\newcommand\Paverage[1]{\varphi\qty(#1)}

\newcommand\rhoConditional[2]{\mathbb{E}_{#1}\left[#2\right]}

\title{\LARGE \bf Lyapunov Stability Analysis for Invariant States of Quantum Systems}

\author{Muhammad F. Emzir, Ian R. Petersen*, and Matthew J. Woolley
	\thanks{A version of this paper has been accepted at 56th {IEEE} Conference on Decision and Control 2017.}%
	\thanks{*This work was supported by the Australian Research Council under grant No. FL110100020 and the AFOSR.}%
	\thanks{Ian R. Petersen is with ANU College of 	Engineering \& Computer Science, The Australian National University, Canberra, Australia {\tt\small ian.petersen@anu.edu.au}}
	\thanks{Muhammad F. Emzir and  Matthew J. Woolley are with School of Engineering and Information Technology, University of New South Wales, ADFA, Canberra, ACT 2600, Australia
		{\tt\small m.emzir@student.adfa.edu.au}}%
}

\begin{document}

\maketitle
\thispagestyle{empty}
\pagestyle{empty}

\begin{abstract}
In this article, we propose a Lyapunov stability approach to analyze the convergence of the density operator of a quantum system. 
In contrast to many previously studied convergence analysis methods for invariant density operators which use weak convergence, in this article we analyze the convergence of density operators by considering the set of density operators as a subset of Banach space. We show that the set of invariant density operators is both closed and convex, which implies the impossibility of having multiple isolated invariant density operators. We then show how to analyze the stability of this set via a candidate Lyapunov operator. 
\end{abstract}

\section{INTRODUCTION}
There are two main approaches to design a feedback controller for a quantum system. The more conventional approach is to compute the feedback input based on measurements of the system, which is known as measurement-based feedback control (MBFC). This method has been well studied in the last two decades \cite{Wiseman1993,Belavkin1999,wiseman2010quantum}. Another approach is to construct the feedback controller as a quantum system that coherently interacts with the controlled system. This method, which is known as coherent feedback control, has recently received considerable interest  \cite{Wiseman1994,James2008,Nurdin2009}.
There are many conditions in which coherent feedback control potentially offers advantages over MBFC; e.g., see \cite{Nurdin2009,Hamerly2012,Yamamoto2014}.\\
There have been many results on analytical tools to analyze the convergence of quantum system dynamics based on  stability analysis of quantum systems subject to measurement, \cite{somaraju2013,Amini2013,amini2014}. 
However, in the absence of measurement, as in the case of coherent control, there are a few of tools available to analyze the stability behavior of quantum dynamical systems. Many established results on linear coherent control of quantum systems are based on stochastic stability criteria involving first and second moments, \cite{James2008}. However, to extend coherent control design beyond the linear case, one should consider more general stability criteria, other than first and second moment convergence.\\
From the classical probability theory point of view, we can consider a system's density operator as a probability measure. Therefore, the convergence of a density operator can be analyzed in a similar way to the convergence of a sequence of probability measures\cite{Parthasarathy1967}. In fact, in the mathematical physics literature, the stability of quantum systems has been analyzed using quantum Markov semigroups via the quantum analog of probability measure convergence \cite{Fagnola2001,Fagnola2003}.  In essence, \cite{Fagnola2003} establishes conditions of the existence of invariant states, as well as convergence to these states given that the invariant state $\rho$ is faithful. That is, for any positive operator $A$, $\trace\qty(A \rho) = 0$ if and only if $A=0$.
\\ 
In the classical control theory, on the other hand, the Lyapunov approach is one of the fundamental tools to examine the stability of classical dynamical systems without solving the dynamic equation \cite{khalil2002nonlinear}. 
There are some important results on the stability of invariance density operator convergence in the Schr\"{o}dinger picture \cite{Ticozzi2008,Wang2010}. In this scheme, Lyapunov analysis is often used, where the Lyapunov function is defined as a function of the density operators.
\\
Recently, \cite{Pan2014} extended results on quantum Markov semigroup invariant state analysis to the Heisenberg picture, which is closely related to Lyapunov stability analysis in the classical setting. Stability analysis in the Heisenberg picture as given in \cite{Pan2014} is interesting for two reasons. The first is that since it is considered in the Heisenberg picture, the stability condition derived is easily connected to classical Lyapunov stability analysis, which is preferable for most control theorists. The second is that, while the stability condition is stated in terms of a Lyapunov observable, it leads to the same conclusion as the quantum Markov semigroup convergence. 
\\
The result of \cite{Pan2014} required that for all non-trivial projection operators $P$,  $P\mathbb{L}^\dagger\qty(1-P)\mathbb{L}P \neq 0$, where $\mathbb{L}$ is the coupling operator of the quantum system.
The weakness of this approach is that in many cases, we deal with quantum systems which have invariant density operators which are not faithful. Furthermore, even when the invariant density operators are faithful, validating the inequality given in \cite[Theorem 3, Theorem 4]{Pan2014} for all non-trivial projection operators is not straight forward; see also \cite[Example 4]{Pan2014}.\\
We aim to establish a stability criterion which is similar to Lyapunov stability theory in classical systems to examine the convergence of the system's density operator. We show that if there is a self-adjoint operator that has a strict minimum value at the invariant state and its generator satisfies a particular inequality condition, then we can infer Lyapunov, asymptotic, and exponential stability in both local and global settings. 
\\
We refer the reader to the following monographs for an introduction to quantum probability \cite{bouten2007introduction,fagnola1999quantum}.
\subsection{Notation}
The Identity operator will be denoted by $1$. A class of operators will be denoted by fraktur type face; e.g., the class of bounded linear operators on a Hilbert space $\mathscr{H}$ $\mathfrak{B}\left(\mathscr{H}\right)$. We use $\norm{\cdot}_{\infty}$ to denote the uniform operator norm on $\mathfrak{B}\left(\mathscr{H}\right)$, and $\norm{X}_1 = \trace\qty(\abs{X})$ for any trace-class operator $X$. The set of density operators (positive operators with unity trace) on the Hilbert space $\mathscr{H}$ is denoted by $\mathfrak{S}\qty(\mathscr{H})$.
Bold letters (e.g. $\mathbf{y}$) will be used to denote a matrix whose elements are operators on a Hilbert space. The Hilbert space adjoint is indicated by $^{\ast}$, while the complex adjoint transpose will be denoted by $\dagger$; i.e.,  $\left(\mathbf{X}^{\ast}\right)^{\top} = \mathbf{X}^{\dagger}$. For single-element operators, we will use $*$ and $\dagger$ interchangeably. The commutator matrix of $\mathbf{x}$ and $\mathbf{y}$ is given by $[\mathbf{x}, \mathbf{y} ] = \mathbf{x}\mathbf{y}^\top - \left(\mathbf{y}\mathbf{x}^\top\right)^\top$. 
\section{Preliminaries}
In this section, we will describe some preliminaries that will be used in the later sections. 

\subsection{Closed and Open Quantum System Dynamics}
Here, we review the basic concepts of closed and open quantum system dynamics. This section is adapted from \cite{Emzir2016a}.
For quantum systems, in contrast to classical systems where the state is determined by a set of scalar variables, the \emph{state} of the system is described by a vector in the system's Hilbert space $\mathscr{H}$ with unit norm. Furthermore, in quantum mechanics, physical quantities like the spin of atom, position, and momentum, are described as self-adjoint operators in a Hilbert space. These operators are called observables. An inner product gives the expected values of these quantities. For example, an observable $A$ and a unit vector $\ket{\psi} \in \mathscr{H}$ have lead to the expected value $\bra{\psi}A\ket{\psi}$.
\\
The dynamics of a closed quantum system are described by an observable called the \emph{Hamiltonian} $\mathbb{H}$ which acts on the unit vector $\ket{\psi} \in \mathscr{H}$, as per 
$
\dv{\ket{\psi_t}}{t} = -i \mathbb{H} \ket{\psi_t},
$
which is  known as the Schr\"{o}dinger equation. The evolution of the unit vector $\ket{\psi} \in \mathscr{H}$ can be described by a unitary operator $U_t$, where $\ket{\psi_t} = U_t \ket{\psi_0}$. Accordingly, the Schr\"{o}dinger equation can be rewritten as
\begin{dmath}
	\dv{U_t}{t} = -i \mathbb{H} U_t. \label{eq:ClosedUnitraryEvolution}
\end{dmath}
From this equation, any system observable $X$ will evolve according to $X_t = U_t^\ast X U_t$, satisfying
\begin{dmath}
	\dv{X_t}{t} = -i \commutator{X_t}{\mathbb{H}}, \label{eq:ClosedSystemHeisenberg}
\end{dmath}
which is called the Heisenberg equation of motion for the observable $X$.\\
An open quantum system is a quantum system which interacts with other quantum mechanical degrees of freedom.
An open quantum system $\mathcal{P}$ can be characterized by a triple $\left(\mathbb{S},\mathbb{L},\mathbb{H}\right)$, with  Hamiltonian $\mathbb{H}$, coupling operator $\mathbb{L}$ and  scattering matrix $\mathbb{S}$ which are operators on the \emph{system}'s Hilbert space $\mathscr{H}$. 
Let $\mathbf{A}_t = \qty[A_{1,t} \; \cdots \; A_{n,t}]$, be a vector of annihilation operators defined on distinct copies of the Fock space $\Gamma$ \cite{nurdin2014quantum}. 
For an open quantum system interacting with $n$ channels environmental fields, the total Hilbert space will be given as $\tilde{\mathscr{H}} = \mathscr{H} \otimes \Gamma_n$, where $\mathscr{H}$ is the system Hilbert space, and $\Gamma_n = \Gamma^{\otimes^n}$ is $n$ copies of the single channel Fock space $\Gamma$. 
Notice that in the linear span of coherent states, the Fock spaces $\Gamma_{i}$ $i=1,\cdots,n$ possesses a continuous tensor product. For any time interval $0\leq s < t $, the Fock space $\Gamma_i$ can be decomposed into  
$
	\Gamma_i = \Gamma_{i,s]}\otimes \Gamma_{i,[s,t]}\otimes \Gamma_{i,[t}. \label{eq:GammaDecompose}
$, \cite[pp. 179-180]{Parthasarathy1992}.
Therefore, we can write $\tilde{\mathscr{H}}_{,t]} \equiv \tilde{\mathscr{H}}_{\left[0,t\right]} = \mathscr{H} \otimes {\Gamma_n}_{\left[0,t\right]}$, and $\tilde{\mathscr{H}}_{[t} =   \Gamma_{n,[t}$.
Each annihilation operator $A_{i,t}$ represents a single channel of quantum noise input. ${\Lambda}$ is a scattering operator between channels.  Both $\mathbf{A}_t$ and $\mathbf{A}_t^\ast$ construct a quantum version of Brownian motion processes, while on the other hand ${\Lambda}$ can be thought as a quantum version of a Poissonian process \cite{Parthasarathy1992}.
In a similar way to the unitary operator evolution in the closed quantum system \eqref{eq:ClosedUnitraryEvolution}, we can also derive the unitary operator evolution for an open quantum system. In contrast to the closed quantum system unitary evolution \eqref{eq:ClosedUnitraryEvolution}, the interaction with the environment leads to randomness in the unitary evolution of an open quantum system $\mathcal{P}$ as follows \cite[Corollary 26.4]{Parthasarathy1992}:
\begin{dmath}
	dU_t =   \left[\text{tr}\left[\left(\mathbb{S} - \mathbf{I}\right)d{\Lambda}^\top_t\right] + d\mathbf{A}^{\dagger}_t \mathbb{L}  - \mathbb{L}^{\dagger}\mathbb{S} d\mathbf{A}_t - \left(\dfrac{1}{2}\mathbb{L}^{\dagger}\mathbb{L}+i\mathbb{H}\right)dt\right]U_t,{\quad U_0 = 1}. \label{eq:UnitaryEvolution}
\end{dmath}
In the context of open quantum system dynamics, any system observable $X$ will evolve according to
\begin{dmath}
	{X_t = j_t\qty(X) \equiv  U^{\dagger}_t\left(X \otimes 1 \right)U_t}, \label{eq:jt}
\end{dmath}
where $1$ is identity operator on $\Gamma_n$. Correspondingly, as an analog of \eqref{eq:ClosedSystemHeisenberg}, for an open quantum system, the corresponding Heisenberg equation of motion for a system operator $X$ is given by \cite{gough2009series},
\begin{dmath}
	dX_t = \mathcal{G}\qty(X_t) dt
	+ d\mathbf{A}^{\dagger}_t\mathbb{S}^{\dagger}_t\left[X_t,\mathbb{L}_t\right]^\top + \left[\mathbb{L}^{\dagger}_t,X_t\right]^\top\mathbb{S}_t d\mathbf{A}_t 
	+ \text{tr}\left[\left(\mathbb{S}^{\dagger}_t X_t \mathbb{S}_t - X_t\right)d{\Lambda}_t^{\top}\right], \label{eq:QSDE_X}
\end{dmath}
where all operators evolve according to  \eqref{eq:jt}; i.e. $\mathbb{L}_t = U^{\dagger}_t\left(\mathbb{L} \otimes 1 \right)U_t$, and $\mathcal{G}\qty(X_t)$ is the quantum Markovian generator for $X_t$ given by
\begin{dmath}
	\mathcal{G}\qty(X_t) = -i\left[X_t,\mathbb{H}_t\right] + \frac{1}{2} \mathbb{L}_t^{\dagger} \left[X_t,\mathbb{L}_t\right]^\top + \frac{1}{2} \left[\mathbb{L}_t^{\dagger},X_t\right]^\top\mathbb{L}_t. \label{eq:QuantumMarkovianGeneratorForX}
\end{dmath}
We call equation \eqref{eq:QSDE_X}  the {QSDE} for the system observable $X$. 

\section{Quantum Dynamical Semigroups and Their Convergences}\label{sec:PreliminaryLyapunovStability}
In this section, we will describe some preliminaries that will be used in the later sections. The following definitions are the basic notions in quantum probability and quantum dynamical semigroups (QDS); see \cite[Chapter 1]{davies1976quantum}. \\
We recall that a von Neumann algebra is a $\ast-$ subalgebra of $\mathfrak{B}\qty(\mathscr{H})$ which contains the identity $1$ and is closed in the normal topology.
\begin{definition}
	Let $\mathscr{A}$ be a von Neumann algebra. A linear functional $\varphi$ is called a state on $\mathscr{A}$, $\varphi : \mathscr{A} \rightarrow \mathbb{C}$ if it is positive i.e., $\varphi\qty(A^\ast A) \geq 0, \forall A \in \mathscr{A}$, and normal; i.e., $\varphi\qty(1) = 1$. 
\end{definition}
Any positive linear functional $\omega$ on $\mathscr{A}$, is called normal if $\sup_n \omega\qty(X_n) =  \omega\qty(\sup_n X_n)$, where $\qty{X_n}$ is an upper bounded increasing net of self adjoint operators.
Notice that a linear functional $\omega$ is  normal if there is a unity trace operator $\rho$ such that $\omega(X) = \trace\qty(\rho X)$. From this viewpoint, a unit element $\ket{u} \in \mathscr{H}$ and a density operator $\rho$ can also be considered as states on $\mathfrak{B}\qty(\mathscr{H})$, by considering the following linear functionals, $\varphi\qty(X) =  \bra{\psi}X\ket{\psi}$, and $\varphi\qty(X) = \trace\qty(\rho X)$.
\\
If $\rho \in \mathfrak{S}\qty(\mathscr{H})$ is the initial density operator of the system and $\Psi \in  \mathfrak{S}\qty(\Gamma_n)$ is the initial density operator of the environment, then for any bounded system observable $X$, the quantum expectation of $X_t$ in \eqref{eq:jt} is given by $\trace{ \qty(X_t \qty(\rho \otimes \Psi))} \equiv \Paverage{X_t}$. Let $\Psi_{[t} \in  \mathfrak{S}(\tilde{\mathscr{H}}_{[t})$ be a density operator on $\tilde{\mathscr{H}}_{[t}$. We can define $\rhoConditional{t]}{\cdot} : \mathfrak{B}\qty(\tilde{\mathscr{H}}) \rightarrow \mathfrak{B}(\tilde{\mathscr{H}}_{t]})$ as follows:
\begin{equation}
\rhoConditional{t]}{Z} \otimes 1 \equiv \condExpectA{Z}{\mathfrak{B}\qty(\tilde{\mathscr{H}}_{t]})\otimes 1 } {, \forall Z \in \mathfrak{B}\qty(\tilde{\mathscr{H}})},
\label{eq:PartialTraceExpect}
\end{equation}
where $1$ is identity operator on $\tilde{\mathscr{H}}_{[t}$ and $ \condExpectA{Z}{\mathfrak{B}\qty(\tilde{\mathscr{H}}_{t]})\otimes 1 }$ is a quantum conditional expectation; see also \cite[Proposition 16.6, Excercise 16.10, 16.11]{Parthasarathy1992} and \cite[Example 1.3]{fagnola1999quantum} for the existence of $\condExpectA{Z}{\mathfrak{B}\qty(\tilde{\mathscr{H}}_{t]})\otimes 1 }$. In our case, we will frequently consider $\rhoConditional{0]}{j_t\qty(X)}$, when the quantum expectation of $j_t\qty(X)$ is marginalized with respect to the system Hilbert space $\mathscr{H}$. 
\begin{definition}\cite{fagnola1999quantum}
	A {QDS} on a von Neumann algebra $\mathscr{A}$ is a family of bounded linear maps $\qty{\mathcal{T}_t, t\geq 0}$ with the following properties
	\begin{enumerate}
		\item $\mathcal{T}_0\qty(A) = A$, for all $A \in \mathscr{A}$.
		\item $\mathcal{T}_{s+t}\qty(A) =  \mathcal{T}_{s}\qty(\mathcal{T}_{t}\qty(A))$, for all $s,t \geq 0$, and $A \in \mathscr{A}$.
		\item $\mathcal{T}_{t}$ is a completely positive mapping for all $t\geq 0$.
		\item $\mathcal{T}_{t}$ is normally continuous for all $t\geq 0$.
	\end{enumerate} 
\end{definition}
Using the conditional expectation in \eqref{eq:PartialTraceExpect}, we observe that there exists a one-parameter semigroup $\mathcal{T}_t : \mathfrak{B}\qty(\mathscr{H}) \rightarrow \mathfrak{B}\qty(\mathscr{H})$ given by 
$
{\mathcal{T}_t\qty(X) = \rhoConditional{0]}{j_t\qty(X)} }.
$
The generator of this semigroup $\mathcal{L}\qty(X) : \mathscr{D}\qty(\mathcal{L}) \rightarrow \mathfrak{B}\qty(\mathscr{H})$ is given by
\begin{dmath}
	{\mathcal{L}\qty(X) = \lim\limits_{t \downarrow 0} \dfrac{\mathcal{T}_t\qty(X) - X}{t} \quad \forall X \in \mathscr{D}\qty(\mathcal{L})},
\end{dmath}
By using the quantum conditional expectation property $\varphi\qty(A) = \varphi\qty(\condExpectA{A}{\mathscr{B}}), \forall A \in \mathscr{A}$ and \eqref{eq:PartialTraceExpect}, we obtain $\Paverage{\mathcal{T}_t\qty(X)\otimes 1} = \Paverage{X_t}$. Therefore,
$
{\mathcal{T}_t\qty(\mathcal{L}\qty(X)) = \mathcal{L}\qty(\mathcal{T}_t\qty(X)) = \rhoConditional{0]}{\mathcal{G}\qty(X_t)}}.
$
Note that there also exists a one-parameter semigroup $\mathcal{S}_t$ such that, $\Paverage{X_t} = \Paverage{\mathcal{T}_t\qty(X) \otimes 1} = \trace\qty(X \mathcal{S}_t\qty(\rho))$. Explicitly, it can be defined as
$
\mathcal{S}_t\qty(\rho) \equiv \trace_{\Gamma_n}\qty(U_t(\rho \otimes \Psi)U_t^\dagger),
$
where $\trace_{\Gamma_n}\qty(\cdot)$ is the partial trace operation over $\Gamma_n$. Let us write $\rho_t \equiv \mathcal{S}_t\qty(\rho)$. The generator of this semigroup is the master equation \cite{fagnola1999quantum}
\begin{equation}
\mathcal{L}_{\ast}\qty(\rho_t) \equiv
-i\left[\mathbb{H},\rho_t\right] + \mathbb{L}^{\top}\rho_t\mathbb{L}^{\ast} - \frac{1}{2}\mathbb{L}^{\dagger}\mathbb{L}\rho_t - \frac{1}{2}\rho_t\mathbb{L}^{\dagger}\mathbb{L}. \label{eq:MasterEquation}
\end{equation}
We restrict our discussion to the case where both of the semigroups $\mathcal{T}_t$ and $\mathcal{S}_t$ are uniformly continuous. For our development, we will also require the following definitions \cite{Fagnola2001}:
\begin{definition}
	A density operator $\rho$ is called invariant for a {QDS} $\mathcal{T}_t$ if for all $A \in \mathfrak{B}\qty(\mathscr{H})$, $\trace\qty(\rho \mathcal{T}_t\qty(A))  = \trace\qty(\rho A)$.
\end{definition}
\begin{definition}\label{def:WeakConvergence}
	A sequence of density operators $\qty{\rho_n}$ is said to converge \emph{weakly} to $\rho \in \mathfrak{S}\qty(\mathscr{H})$ if for all $A \in \mathfrak{B}\qty(\mathscr{H})$,
$
	\lim\limits_{n\rightarrow \infty} \trace\qty(\rho_n A) = \trace\qty(\rho A). 
$
\end{definition}
We write the limit of a weakly converge sequence $\qty{\rho_n}$ as $w-\lim\limits_{n\rightarrow \infty} \rho_n$ and $\rho_n \xrightarrow{w} \rho$. 
We recall that the set of trace-class Hilbert space operators $\mathfrak{I}_1\qty(\mathscr{H})$ with norm $\norm{\cdot}_1$ is a Banach space \cite[Prop 9.13]{Parthasarathy1992}.
Using the metric induced by the norm $d\qty(\rho_A,\rho_B) \equiv \norm{\rho_A-\rho_B}_1$, we refer a closed ball with center $\rho_{\ast}$ and radius $\epsilon$ to the set
\begin{dmath}
	\mathscr{B}_{\epsilon}\qty(\rho_{\ast}) = \qty{\rho \in \mathfrak{S}\qty(\mathscr{H}) : \norm{\rho-\rho_{\ast}}_1\leq \epsilon}. \label{eq:BallCenterAtRhoStarRadiusEpsilon}
\end{dmath}
The normalized version of the distance $d\qty(\rho_A,\rho_B)$ is also known as the Kolmogorov distance in quantum information community \cite{Fuchs1999,MichaelA.Nielsen2001,IngemarBengtsson2017}.We will also refer neighborhood $\mathscr{N}$ of $\rho_\ast$ to a union of balls \eqref{eq:BallCenterAtRhoStarRadiusEpsilon} with various center points and $\mathscr{B}_{\epsilon}\qty(\rho_{\ast}) \subseteq \mathscr{N}$ for some $\epsilon$.
The following proposition shows a basic fact regarding the completeness of the class of density operators under $\norm{\cdot}_1$.
\begin{proposition}\label{prp:SHisCompleteSubspaceOfI1}
	The class of density operators on the Hilbert space $\mathscr{H}$, $\mathfrak{S}\qty(\mathscr{H})$ is a closed subset of the Banach space $\qty(\mathfrak{I}_1\qty(\mathscr{H}),\norm{\cdot}_1)$.
\end{proposition}
\begin{proof}
	First we recall that a subset of a complete metric space is closed if and only if it is complete \cite[Prop 6.3.13]{Morris2016}. Therefore, we need to show that every Cauchy sequence of density operators $\qty{\rho_n}$ converges to a density operator $\rho_\ast$. Since $\mathfrak{S}\qty(\mathscr{H}) \subseteq \mathfrak{I}_1\qty(\mathscr{H})$, which is a Banach space with respect to the norm $\norm{\cdot}_1$, then $\qty{\rho_n}$ converges to an element in $\mathfrak{I}_1\qty(\mathscr{H})$; i.e., $\rho_\ast \in \mathfrak{I}_1\qty(\mathscr{H})$. Therefore according to the definition of the class of density operators, it remains to show that $\rho_\ast $ is positive and has unity trace. The limit $\rho_\ast$ is positive, since if it is non-positive, then there exists $\epsilon >0$ such that for all $n$, $0 < \epsilon < \norm{ \rho_n - \rho_\ast}_\infty$. However, as $n \rightarrow \infty$,
	\begin{align*}
	0 < \epsilon < \norm{ \rho_n - \rho_\ast}_\infty \leq  \norm{ \rho_n - \rho_\ast}_1 = 0,
	\end{align*}
	which is a contradiction. 
	The limit $\rho_\ast$ also has unit trace by the following argument. Since $\rho_n$ converge to $\rho_\ast$, then for any $\epsilon >0$ there is $n$ such that 
	\begin{align*}
	1 = \norm{ \rho_n }_1 \leq \norm{ \rho_n - \rho_\ast}_1 + \norm{ \rho_\ast}_1 \leq \epsilon + \norm{\rho_\ast}_1.
	\end{align*}
	However, we notice also that for any $\epsilon >0$, there is an $N_\epsilon \in \mathbb{N}$, such that for every $n,m \geq N_\epsilon$, $\norm{\rho_n - \rho_m}_1 < \epsilon$. 
	Fix $n$. Then we have $\norm{\rho_n}_1 \leq \norm{\rho_n - \rho_{N_\epsilon}}_1 +\norm{\rho_{N_\epsilon}}_1$. Taking the limit as $n$ approaches infinity, we obtain 
	\begin{dmath*}
		\norm{\rho_{\ast}}_1 \leq  \lim\limits_{n\rightarrow \infty} \norm{\rho_n}_1 + \norm{\rho_n - \rho_{\ast}}_1 
		\leq  \lim\limits_{n\rightarrow \infty} \norm{\rho_n - \rho_{N_\epsilon}}_1 +\norm{\rho_{N_\epsilon}}_1 + \norm{\rho_n - \rho_{\ast}}_1 
		\leq \epsilon + 1.
	\end{dmath*}
	Since $\epsilon$ can be chosen arbitrarily, then $\norm{\rho_{\ast}}_1 = 1$. Therefore, $\rho_{\ast}$ is indeed a density operator.
\end{proof}
\section{Lyapunov Stability Criterion for The Invariance Set of Density Operators}
In this section, we will introduce a Lyapunov stability notion for the set of system invariant density operators. Before we define the stability condition in the following proposition, we first show that the set of invariant density operators of the {QDS} $\mathcal{T}_s$ is both closed and convex. 
\begin{proposition}\label{prp:InvariantDensityOperatorConvexClosed}
	The set of invariant density operators $\mathscr{C}_\ast$ is convex and closed in $\qty(\mathfrak{I}_1\qty(\mathscr{H}),\norm{\cdot}_1)$.
\end{proposition}
\begin{proof}
	The convexity of $\mathscr{C}_\ast$ follows directly from the fact that for any $\rho_1,\rho_2 \in \mathscr{C}_\ast$, then for any $\lambda \in \qty[0,1]$, $\trace\qty(\qty(\mathcal{S}_t\qty(\lambda \rho_1 + (1-\lambda)\rho_2) - \qty(\lambda \rho_1 + (1-\lambda)\rho_2))A) =  0$, for all $A \in \mathfrak{B}\qty(\mathscr{H})$, and $t \geq 0$.
	Notice that since $\mathscr{C}_\ast$  is convex, the closedness of  $\mathscr{C}_\ast$ on $\qty(\mathfrak{I}_1\qty(\mathscr{H}),\norm{\cdot}_1)$ is equivalent to the closedness of $\mathscr{C}_\ast$  in the weak topology  \cite[Thm III.1.4 ]{Conway1985}.
	To show that $\mathscr{C}_\ast$ is closed in the weak topology, suppose that $\qty{\rho_n}$ is a net in $\mathscr{C}_\ast$ that converges weakly to $\rho_\ast$, $\rho_n \xrightarrow{w} \rho_\ast$. Then we have to show that $\rho_\ast \in \mathscr{C}_\ast$.  We observe that, the linearity of the semi-group $\mathcal{S}_t$ implies that for all $A \in \mathfrak{B}\qty(\mathscr{H})$, $\rho_n$ in the net $\qty{\rho_n}$, and $t\geq 0$
	\begin{dmath*}
		\trace\qty(\qty(\mathcal{S}_t\qty(\rho_\ast)-\rho_\ast)A) = \trace\qty(\qty(\mathcal{S}_t\qty(\rho_\ast - \rho_n))A) + \trace\qty(\qty(\mathcal{S}_t\qty(\rho_n) - \rho_n)A) + \trace\qty(\qty(\rho_n - \rho_\ast)A)\\ = \trace\qty(\qty(\mathcal{S}_t\qty(\rho_\ast - \rho_n))A)  + \trace\qty(\qty(\rho_n - \rho_\ast)A) = 
		\trace\qty(\qty(\rho_\ast - \rho_n)\qty(\mathcal{T}_t\qty(A) - A)).
	\end{dmath*}
	Since $\mathcal{T}_t\qty(A) \in \mathfrak{B}\qty(\mathscr{H})$, then for any $\epsilon >0$, there exists a $\rho_m \in \qty{\rho_n}$ such that $\abs{\trace((\rho_\ast - \rho_m)(\mathcal{T}_t\qty(A) - A))} < \epsilon$.
	However,  $\epsilon>0$ can be selected arbitrarily , therefore, $\rho_\ast  \in \mathscr{C}_\ast$.
\end{proof}
\begin{remark}
	The last proposition implies that in any quantum system, it is impossible to have multiple isolated invariant density operators, even for the case of finite dimensional quantum systems. This phenomenon is unique to quantum systems since classical dynamics can have multiple isolated equilibrium points; see for example \cite[\textsection 2.2 ]{khalil2002nonlinear}. The convexity of $\mathscr{C}_\ast$ has also been derived in \cite{Schirmer2010} for the finite dimensional case. 
\end{remark}
In what follows, we will examine the convergence to the set of invariant density operators in a Banach space $\qty(\mathfrak{I}_1\qty(\mathscr{H}),\norm{\cdot}_1)$. The distance between a point $\sigma \in \mathfrak{S}\qty(\mathscr{H})$ and the closed convex set $\mathscr{C}_\ast \subseteq \mathfrak{S}\qty(\mathscr{H})$ can  be naturally defined by
\begin{dmath}
	d\qty(\sigma,\mathscr{C}_\ast) = \inf_{\rho \in \mathscr{C}_\ast} \norm{\sigma-\rho}_1. \label{eq:dtoC_ast}
\end{dmath}
We define the following stability notions:
\begin{definition}\label{def:StabilityOfDensityOperator}
	Let $\mathscr{C}_{\ast} \subset \mathfrak{S}\qty(\mathscr{H}) $ be a convex set of invariant density operators of a quantum system $\mathcal{P}$. Suppose that $\mathscr{N} \subset \mathfrak{S}\qty(\mathscr{H})$, where $\mathscr{C}_{\ast}$ is a strict subset of $\mathscr{N}$, and the system is initially at density operator $\rho \in \mathscr{N}$. Then, we say the closed convex set of invariant density operators $\mathscr{C}_{\ast}$ is,
	\begin{enumerate}
		\item Lyapunov stable if for every $\varepsilon >0$,  there exists $\delta(\varepsilon) >0$ such that $d\qty(\rho,\mathscr{C}_{\ast}) < \delta(\varepsilon)$ implies $d\qty(\rho_t,\mathscr{C}_{\ast}) < \varepsilon$ for all $t\geq 0$.
		\item Locally asymptotically stable, if it is Lyapunov stable, and there exists $\delta >0$, such that $d\qty(\rho,\mathscr{C}_{\ast}) < \delta$ implies $\lim\limits_{t\rightarrow\infty} d\qty(\rho_t,\mathscr{C}_{\ast}) = 0$.
		\item Locally exponentially stable, if there exists $\beta,\gamma,\delta >0$  such that $d\qty(\rho,\mathscr{C}_{\ast}) < \delta$ implies, $d\qty(\rho_t,\mathscr{C}_{\ast}) \leq \beta d\qty(\rho,\mathscr{C}_{\ast}) \exp\qty(-\gamma t)$ for all $t\geq 0$.
	\end{enumerate}
	If $\mathscr{N} = \mathfrak{S}\qty(\mathscr{H})$, such that $\delta$ can be chosen arbitrarily in 2) and 3), we say $\mathscr{C}_{\ast}$ is a globally asymptotically, or exponentially stable respectively.
\end{definition}
Before we prove the main result, we need to establish the following facts; see \cite{Emzir2017b} for the proof.
\begin{lemma}\cite{Emzir2017b}\label{lem:LowerBound}
	Suppose there exists a self-adjoint operator $A \in \mathfrak{B}\qty(\mathscr{H})$  where spectrum of $A$ is non decreasing such that for a closed convex set of density operators $\mathscr{C}_\ast$, and a neighborhood $\mathscr{N}$ of $\mathscr{C}_\ast$, 
	\begin{align}
	\inf_{\rho_\ast \in \mathscr{C}_\ast}\trace\qty(A\qty(\rho - \rho_{\ast})) > 0, \; \forall \rho \in \mathscr{N}\backslash \mathscr{C}_\ast. \label{eq:StrictMinimaOnNeighbourhood}
	\end{align} 
	Then there exists $\kappa>0$ such that $\kappa d\qty(\rho,\mathscr{C}_\ast)^2 \leq \inf_{\rho_\ast \in \mathscr{C}_\ast}\trace\qty(A\qty(\rho - \rho_{\ast}))$, for all $\rho \in \mathscr{N} \backslash \mathscr{C}_\ast$.
\end{lemma}

\begin{lemma}\cite{Emzir2017b} \label{lem:WeakConverge}
	Let $\mathscr{C}_\ast$ be a closed convex set of invariant density operators and $\mathscr{N}$ be a neighborhood of $\mathscr{C}_\ast$. Suppose $\qty{\rho_n}$ is a sequence of density operators in $\mathscr{N} \backslash \mathscr{C}_\ast$. If there exists a self-adjoint operator $A \in \mathfrak{B}\qty(\mathscr{H})$ satisfying condition in Lemma \ref{lem:LowerBound} and $\lim\limits_{n\rightarrow \infty}\trace\qty(A\qty(\rho_n - \rho_\ast) ) = 0$ for a $\rho_\ast \in \mathscr{C}_\ast$, then  $\lim\limits_{n\rightarrow \infty} d\qty(\rho_n,\mathscr{C}_\ast) = 0$.
\end{lemma}

The following theorem is the main result of this article, which relates the stability notions defined above to an inequality for the generator of a candidate Lyapunov operator. 
\begin{theorem}\label{thm:LyapunovStability}
	Let $V \in \mathfrak{B}\qty(\mathscr{H})$ be a self-adjoint operator with non decreasing spectrum value such that
		\begin{dmath}
		{\inf_{\rho_\ast \in \mathscr{C}_\ast}\trace\qty(V\qty(\rho - \rho_{\ast})) > 0}, {\forall \rho \in \mathfrak{S}\qty(\mathscr{H}) \backslash  \mathscr{C}_\ast}.
		\label{eq:V_rho_ast_0}
	\end{dmath}
	where $C$ is a real constant. Using the notation of Definition \ref{def:StabilityOfDensityOperator} 
	\begin{enumerate}
		\item If
		\begin{equation}
		\trace\qty(\mathcal{L}\qty(V)\rho) \leq 0, \; {\forall \rho \in \mathscr{N}\backslash \mathscr{C}_\ast}, \label{eq:LyapunovStability}
		\end{equation}
		then $\mathscr{C}_\ast$ is Lyapunov stable.
		
		\item If 
		\begin{equation}
		\trace\qty(\mathcal{L}\qty(V)\rho) < 0, \; {\forall \rho \in \mathscr{N}\backslash \mathscr{C}_\ast},
		\label{eq:AsymptoticallyStableCondition}
		\end{equation}
		then $\mathscr{C}_\ast$ is locally asymptotically stable.
		
		\item If there exists $\gamma >0$ and $\zeta \in \mathbb{R}$ such that
		\begin{equation}
		\trace\qty(\mathcal{L}\qty(V)\rho) \leq -\gamma \trace\qty(V\rho) + \zeta <0  \; {\forall \rho \in \mathscr{N}\backslash \mathscr{C}_\ast}, \label{eq:ExponentiallyStableCondition}
		\end{equation}
		then $\mathscr{C}_\ast$ is locally exponentially stable.
	\end{enumerate}
\end{theorem}
\begin{proof}
	Let us begin by proving the first part. Suppose $\varepsilon >0$ is selected. Then, we can take $\varepsilon' \in (0,\varepsilon]$ such that $\mathscr{B}_{\varepsilon'}\qty(\mathscr{C}_{\ast}) \subseteq \mathscr{N}$. Observe that by \eqref{eq:V_rho_ast_0} and Lemma \ref{lem:LowerBound}, there exists a $\kappa >0$ such that for any $\rho \in  \mathscr{B}_{\varepsilon'}\qty(\mathscr{C}_{\ast}) $
	\begin{dmath}
		\kappa d\qty(\rho,\mathscr{C}_\ast)^2 \leq \inf_{\rho_\ast \in \mathscr{C}_\ast} \trace\qty(V\qty(\rho-\rho_\ast)). \label{eq:NBetaWeaklyCompact}
	\end{dmath}
	Let $V_\ast = \sup_{\rho_\ast \in \mathscr{C}_\ast} \trace\qty(V\rho_\ast)$.
	Therefore, if we select $V_\ast < \beta < V_\ast + \kappa (\epsilon')^2$ then the set $\mathscr{N}_{\beta} = \qty{ \rho \in \mathscr{N} : \trace\qty(V\rho)\leq \beta}$
	is a strict subset of $\mathscr{B}_{\varepsilon'}\qty(\rho_{\ast})$.
	Furthermore, since $\beta > V_\ast$ and $\trace\qty(V\rho) \leq V_\ast + \norm{V}_\infty d\qty(\rho,\mathscr{C}_\ast)$, selecting $\delta < \qty(\beta - V_\ast)/\norm{V}_\infty$ implies $\trace\qty(V\rho) < \beta$ for all $\rho \in \mathscr{B}_{\delta}\qty(\mathscr{C}_{\ast})$.
	Therefore, we have the following relation
	\begin{dmath*}
		{\mathscr{B}_{\delta}\qty(\mathscr{C}_{\ast}) \subset \mathscr{N}_{\beta} \subset \mathscr{B}_{\varepsilon'}\qty(\mathscr{C}_{\ast})}.
	\end{dmath*}
	Therefore, $\rho \in \mathscr{B}_{\delta}\qty(\mathscr{C}_{\ast})$ implies $\rho \in \mathscr{N}_{\beta}$.  Since $\trace\qty(\mathcal{L}\qty(V)\rho) \leq 0$ for all $\rho \in \mathscr{N}\backslash \mathscr{C}_\ast$, if  system density operator $\rho$ is initially in $\mathscr{N}_{\beta}$, then the expected value of operator $V$ will be non-increasing, $
	{\trace\qty(V\rho_t) \leq \trace\qty(V\rho) \leq \beta, \forall t \geq 0}.$
	This implies that $\rho_t \in \mathscr{N}_{\beta}, \forall t \geq 0$, which shows $\rho_t \in \mathscr{B}_{\varepsilon'}\qty(\mathscr{C}_{\ast})$. Furthermore, this last implication implies that if initially $d\qty(\rho,\mathscr{C}_{\ast}) < \delta(\varepsilon)$, then $d\qty(\rho_t,\mathscr{C}_{\ast}) < \varepsilon$ for all $t \geq 0$.
	\\
	For the second part, using the same argument as in the previous part, we may choose $\delta>0$ such that initially $\rho \in \mathscr{B}_{\delta}\qty(\mathscr{C}_\ast) \subset \mathscr{N}_\beta \subset \mathscr{N}$, for some $\beta > V_\ast$. Therefore, since  $\trace\qty(\mathcal{L}\qty(V)\rho) < 0$ for all $\rho \in \mathscr{N}\backslash \mathscr{C}_\ast$, $\trace\qty(V\rho_t)$ is monotonically decreasing. Therefore $\trace\qty(V\rho_t)<\trace\qty(V\rho_s) < \trace\qty(V\rho) < \beta$ for any $0 < s < t$. Hence $\rho_s,\rho_t$ also belongs to $\mathscr{N}_\beta$. This implies that there exists a sequence of density operators $\qty{\rho_n}$, such that $\trace\qty(V\qty(\rho_m - \rho_n)) <0$ for any $m>n$, where $\rho_n \equiv \mathcal{S}_{t_n}\qty(\rho)$ and $0 \leq t_0 <t_1 < \cdots < t_n$, $t_n \rightarrow \infty$ as $n \rightarrow \infty$. Since the spectrum of $V$ is non-decreasing, $\trace\qty(V\rho_n)$ is lower bounded. Hence, there exists a $\rho_c \in \mathscr{N}_\beta$ such that $\lim\limits_{n\rightarrow \infty} \trace\qty(V\qty(\rho_n - \rho_c)) = 0$ and $\trace\qty(V\qty(\rho_c - \rho_n)) <0$ for all $n$. Suppose $\rho_c \notin \mathscr{C}_\ast$. Then for any $s>0$, $\trace\qty(V\mathcal{S}_s\qty(\rho_c)) < \trace\qty(V\rho_c)$. Therefore, there exists an $n$ such that $\trace\qty(V\qty(\rho_c - \rho_n)) >0$, which is a contradiction. Therefore, $\rho_c \in \mathscr{C}_\ast$. Lemma \ref{lem:WeakConverge} and \eqref{eq:dtoC_ast} then imply that $\lim\limits_{n\rightarrow \infty} d\qty(\rho_n, \mathscr{C}_\ast) \leq \lim\limits_{n\rightarrow \infty} \norm{\rho_n - \rho_c}_1 = 0$.
	\\    
	For the global exponentially stable condition, the previous part shows that the negativity of $\trace\qty(\mathcal{L}\qty(V)\rho)$ for all $\rho \in \mathscr{N}\backslash \mathscr{C}_\ast$ implies the existence of a $\rho_c \in \mathscr{C}_\ast$ such that $\lim\limits_{t\rightarrow \infty} \norm{\rho_t-\rho_c }_1=0$. Using the First Fundamental Lemma of Quantum Stochastic Calculus \cite[Prop 25.1]{Parthasarathy1992} to switch the order of the integration and quantum expectation; see also \cite[Prop 26.6]{Parthasarathy1992}, we obtain
	\begin{dmath*}
		\trace\qty(V\rho_t) - \trace\qty(V\rho_s)  = \Paverage{V_t} - \Paverage{V_s} =  \Paverage{\int_{s}^{t} \mathcal{G}\qty(V_{\tau}) d\tau}  = \int_{s}^{t} \Paverage { \mathcal{G}\qty(V_{\tau})} d\tau.
	\end{dmath*}
	Therefore, by \eqref{eq:ExponentiallyStableCondition} we obtain 
	\begin{dmath*}
		\trace\qty(V\rho_t) \leq \qty(\trace\qty(V\rho_s) - \dfrac{\zeta}{\gamma})e^{\gamma(s-t)} + \dfrac{\zeta}{\gamma}.
	\end{dmath*}
	Taking $s=0$, there exists $\kappa >0$ such that $\kappa d(\rho_t,\mathscr{C}_\ast)^2 \leq \kappa \norm{\rho_t - \rho_c}_1^2 \leq \trace\qty(V(\rho_t - \rho_c)) \leq \trace(V\rho_t) - \frac{\zeta}{\gamma} \leq \qty(\trace(V\rho) - \frac{\zeta}{\gamma})e^{- \gamma t}$. Consequently, we obtain
	\begin{dmath*}
		k d(\rho_t,\mathscr{C}_\ast)^2  \leq  \qty(\trace(V\rho) -\frac{\zeta}{\gamma})e^{-\gamma t},
	\end{dmath*}
	which completes the proof.	
\end{proof}
\begin{remark}
	In contrast to the stability conditions given in \cite{Pan2014}, we do not require $V$ to be coercive, nor we demand it to commute with the Hamiltonian of the system \cite{Pan2016}. We show in Theorem \ref{thm:LyapunovStability} that less restrictive conditions on both $V$ and $\mathcal{L}(V)$, \eqref{eq:V_rho_ast_0},\eqref{eq:AsymptoticallyStableCondition} are sufficient to guarantee the convergence of the density operator evolution to the set of invariant density operators. 
\end{remark}
\begin{remark}
	We can use Theorem \ref{thm:LyapunovStability} to strengthen many results in the coherent control of quantum systems. In fact, the differential dissipative inequality given in \cite[Thm 3.5]{James2010} and those which is given as an {LMI} in \cite[Thm 4.2]{James2008} explicitly imply global exponential and asymptotic stability conditions, provided that the storage function in \cite{James2010} and in \cite{James2008} have  global minima at the invariant density operator $\rho_\ast$. 
\end{remark}

\subsection{Examples}
To illustrate the application of the Lyapunov stability conditions we have derived, we consider the following examples. 
\begin{example}\label{exam:LinearQuantumSystem}
	Consider a linear quantum system $\mathcal{P}$, with $\mathbb{H} =  (a - \alpha 1)^\dagger (a - \alpha 1), \alpha \in \mathbb{C}$, $\mathbb{L} = \sqrt{\kappa} \qty(a - \alpha1)$, and $\mathbb{S} = 1$. Evaluating the steady state of \eqref{eq:MasterEquation} we know that the invariant density operator is a coherent density operator with amplitude $\alpha$; i.e., $\rho_{\ast} = \ket{\alpha}\bra{\alpha}$. Now, choose the Lyapunov observable $V = \mathbb{H}$. Straightforward calculation of $\mathcal{L}(V)$ using \eqref{eq:QuantumMarkovianGeneratorForX} gives, 
	\begin{equation*}
	\mathcal{L}\qty(V) = -\kappa (N -\frac{1}{2}\qty(\alpha a^\dagger + \alpha^\ast a ) + \abs{\alpha}^2 1). 
	\end{equation*}
	Notice that $\trace\qty(\mathcal{L}\qty(V)\rho)< 0$ for all density operators other than $\rho_{\ast} = \ket{\alpha}\bra{\alpha}$. To verify this inequality, it is sufficient to take $\rho = \ket{\beta}\bra{\beta}$ with $\beta \neq \alpha$. This follows since every state vector $\ket{\psi}$ can be expressed as a limit of infinite sums of coherent state vectors; i.e., the set of coherent state vectors is total in $\mathscr{H}$; see \cite[\textsection 3.5]{Gerry2004}. Therefore, for $\rho_t = \ket{\beta}\bra{\beta}$, we obtain $\trace \qty(V \rho_t) = \abs{\beta}^2 + \abs{\alpha}^2 -\qty(\alpha^\ast \beta+\beta^\ast \alpha)$ and  $\trace\qty(\mathcal{L}\qty(V)\rho_t) = -\kappa \trace \qty(V \rho_t) + \kappa/2 \qty(\alpha^\ast \beta+\beta^\ast \alpha) <0$. Hence, 
	Theorem \ref{thm:LyapunovStability} indicates that the invariant density operator is exponentially stable. 
\end{example}

\begin{example}
	Consider a nonlinear quantum system with zero Hamiltonian and a coupling operator $\mathbb{L} = (a^2 - \alpha^21) $,  where $\alpha$ is a complex constant.
	To find the invariant density operators of this quantum system we need to find the eigenvectors of $a^2$. Without loss of generality, let $\ket{z}$ be one of the eigenvectors of $a^2$, such that $a^2\ket{z} = \alpha^2 \ket{z}$. Expanding $\ket{z}$ in the number state orthogonal basis, we can write, $a^2 \ket{z} = \sum_{n=0}^{\infty}  a^2 c_n \ket{n}$. Therefore, we find that, $\alpha^2 c_{n-2} = c_n \sqrt{n\qty(n-1)}$. By mathematical induction, we have for $n$ even, $c_n = c_0 \alpha^n/\sqrt{n!}$, and for $n$ odd, $c_n = c_1 \alpha^n/\sqrt{n!}$. Therefore, we can write the eigenvector of $a^2$ as,
	\begin{align*}
	\ket{z} =& c_0 \sum_{n=0, n \text{even}}^{\infty} \dfrac{\alpha^n}{\sqrt{n!}} \ket{n} + c_1 \sum_{n=1, n \text{odd}}^{\infty} \dfrac{\alpha^n}{\sqrt{n!}} \ket{n}.
	\end{align*}
	By observing that a coherent vector with magnitude $\alpha$, is given by 
$	\ket{\alpha} = \exp(-\frac{\abs{\alpha}^2}{2})\sum_{n=0}^{\infty} \frac{\alpha^{n}}{\sqrt{n!}} \ket{n},$
	we can write
$
	\ket{z} = c_0 \exp(\frac{\abs{\alpha}^2}{2}) \frac{\qty(\ket{\alpha} + \ket{-\alpha})}{2} + c_1 \exp(\frac{\abs{\alpha}^2}{2}) \frac{\qty(\ket{\alpha} - \ket{-\alpha})}{2}.
$	
	Normalization of $\ket{z}$ shows that $c_0$ and $c_1$ satisfy an elliptic equation,
$
	\abs{c_0}^2 \cosh(\abs{\alpha}^2)+ \abs{c_1}^2 \sinh(\abs{\alpha}^2) = 1.
$
	Therefore, we can write	the solution of $\mathbb{L} \ket{z} = 0 \ket{z}$ by the following set:
$
	\mathscr{Z}_\ast = \qty{\ket{z}\in \mathscr{H}: \ket{z} = C_0\frac{\qty(\ket{\alpha} + \ket{-\alpha})}{2} + C_1\frac{\qty(\ket{\alpha} - \ket{-\alpha})}{2}}, 
$	
	where $C_0 = c_0 \exp(\frac{\abs{\alpha}^2}{2})$ and $C_1 = c_1 \exp(\frac{\abs{\alpha}^2}{2})$. The set of invariant density operators of this quantum system is a convex set $\mathscr{C}_\ast$ which is given by
$
	\mathscr{C}_\ast = \qty{\sum_i \lambda_i \ket{\beta_i}\bra{\beta_i}:   \ket{\beta_i} \in \mathscr{Z}_\ast}, 
$	
	where $\lambda_i \geq 0, \sum_i \lambda_i = 1$. Suppose we select a Lyapunov candidate $V = \mathbb{L}^\dagger \mathbb{L}$. One can verify that $\trace\qty(\rho V) = 0$ for all $\rho$ belonging to $\mathscr{C}_\ast$, and has a positive value outside of this set.
	Straightforward calculation of the quantum Markovian generator of $V$ using \eqref{eq:QuantumMarkovianGeneratorForX} gives us the following
$
	\mathcal{G}\qty(V_t) = - \qty(4 \mathbb{L}_t^\dagger N_t \mathbb{L}_t + 2 V_t).
$
	Outside the set $\mathscr{C}_\ast$, the generator $\mathcal{G}\qty(V_t)$ has a negative value.  Therefore, Theorem \ref{thm:LyapunovStability} implies that the set $\mathscr{C}_\ast$ is globally exponentially stable.\\ 
	Figure \ref{fig:Example3PhaseSpace} illustrates the phase-space of the system corresponding to various initial density operators.  This figure shows that each distinct trajectory converges to a different invariant density operator, all belonging to the set of invariant density operators $\mathscr{C}_\ast$. Moreover, Figure \ref{fig:Example3Lyapunov} shows the Lyapunov operator expected values. This figure also shows that although each trajectory converges to a distinct invariant density operator, their Lyapunov expected values all converge to zero. 
	 
	\begin{figure}[!h]
		\centering
		\includegraphics[width=0.50\textwidth]{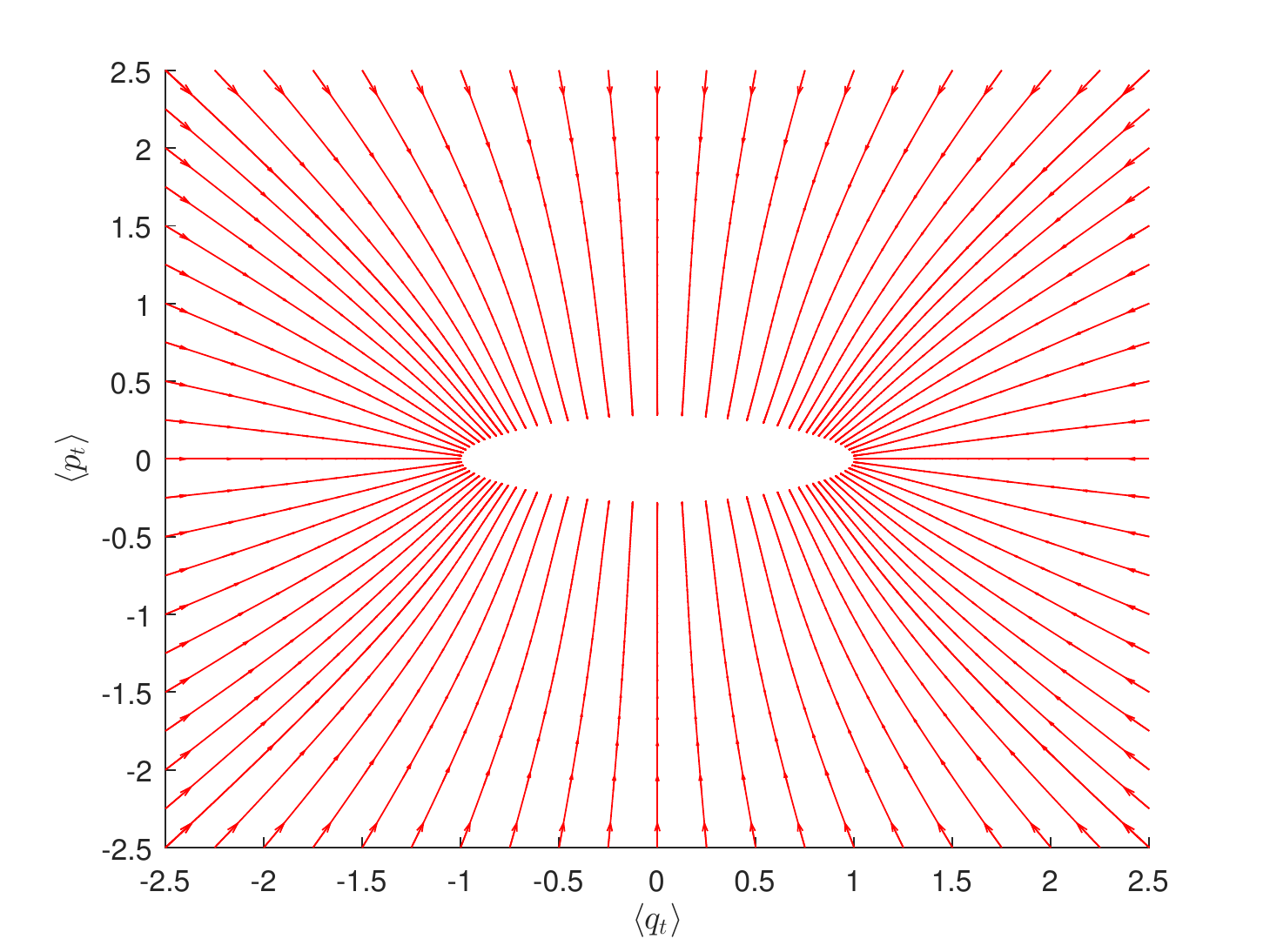}
		\caption[Phase space of the quantum system in Example 2.]{Trajectories of the quantum system in Example 2, simulated using the corresponding master equations. Each line corresponds to a different initial density operator.} 
		\label{fig:Example3PhaseSpace}
	\end{figure}
	\begin{figure}[!h]
	\centering
	\includegraphics[width=0.50\textwidth]{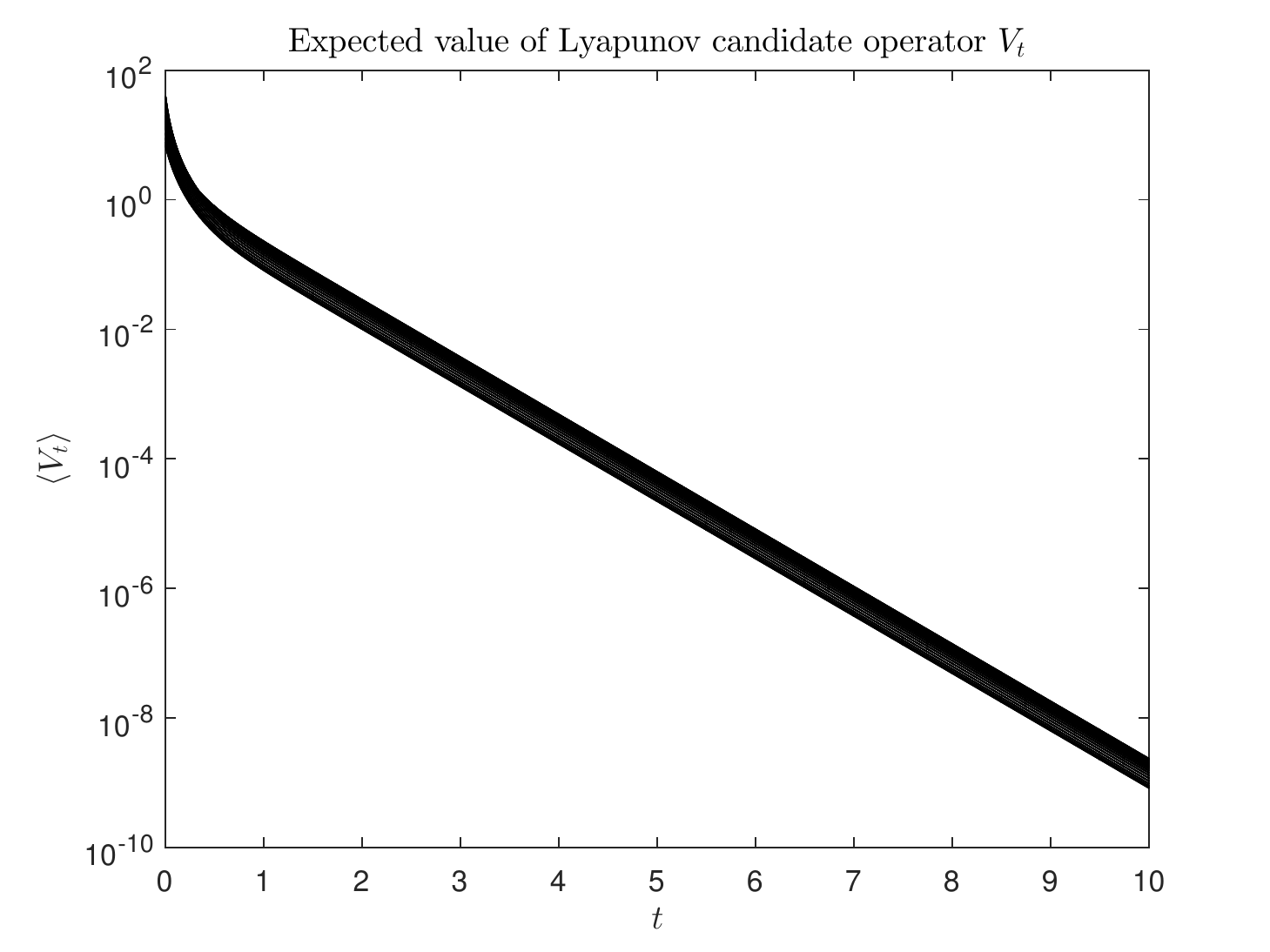}
	\caption[Lyapunov operator expected value of the quantum system in Example 2.]{Lyapunov operator expected value of the quantum system in Example 2.} 
	\label{fig:Example3Lyapunov}
	\end{figure}
\end{example}
\section{Conclusion}
In this article, we have proposed a Lyapunov stability approach for open quantum systems to investigate the convergence of the system's density operator in $\norm{ \cdot }_1$. This stability condition is stronger compared to the finite moment convergence that has been considered for quantum systems previously. 
\\
We have proven that the set of invariant density operators of any open quantum system is both closed and convex. Further, we have shown how to analyze the stability of this set via a Lyapunov candidate operator. 
\\
We have also demonstrated that a quantum system where the generator of its Lyapunov observable is non-negative has at least one invariant density operator.  This connection offers a straightforward approach to verify both the Lyapunov stability condition and the existence of an invariant density operator.
\section{Acknowledgments}
The authors acknowledge discussions with Dr. Hendra Nurdin of UNSW.

\bibliographystyle{IEEEtran}
\bibliography{ReferenceAbbrvBibLatex}
\end{document}